\documentclass[a4paper]{amsart}
\usepackage{amssymb}
\usepackage{amscd}
\usepackage{amsthm}
\usepackage{amsmath}
\usepackage{latexsym}

\theoremstyle{plain}
\newtheorem{theorem}{Theorem}
\newtheorem{corollary}[theorem]{Corollary}
\newtheorem{lemma}[theorem]{Lemma}

\theoremstyle{definition}
\newtheorem{definition}[theorem]{Definition}
\newtheorem{example}[theorem]{Example}

\theoremstyle{remark}

\newtheorem{remark}[theorem]{Remark}

\numberwithin{theorem}{section}

\newcommand{\mspec}[1]{\mbox{\rm{mspec}}(#1)}
\newcommand{\Add}{\mbox{\rm{Add\,}}}

\newcommand{\Filt}{\mbox{\rm{Filt\,}}}

\newcommand{\Hom}[3]{\mbox{\rm{Hom}}_{#1}(#2,#3)}
\newcommand{\Ext}[4]{\mbox{\rm{Ext}}^{#1}_{#2}(#3,#4)}
\newcommand{\Tor}[4]{\mbox{\rm{Tor}}_{#1}^{#2}(#3,#4)}

\newcommand{\rfmod}[1]{\mbox{\rm{mod}--}{#1}}
\newcommand{\rmod}[1]{\mbox{\rm{Mod}--}{#1}}

\newcommand{\ModR}{\text{Mod-}R}

\newcommand{\Ker}[1]{\mbox{\rm{Ker}}(#1)}

\newcommand{\leqk}{${\leq\kappa}$}
\newcommand{\sipka}[1]{\stackrel{#1}\longrightarrow}
\let\ol=\overline
\let\hrp=\upharpoonright

\begin{document}

\title{Approximations and locally free modules}
\author{\textsc{Alexander Sl\' avik and Jan Trlifaj}}
\address{Charles University, Faculty of Mathematics and Physics, Department of Algebra \\
Sokolovsk\'{a} 83, 186 75 Prague 8, Czech Republic}
\email{slavik.alexander@seznam.cz}
\email{trlifaj@karlin.mff.cuni.cz}

\date{\today}
\thanks{Research supported by GA\v CR 201/09/0816.}
\begin{abstract} For any set of modules $\mathcal S$, we prove the existence of precovers (right approximations) for all classes of modules of bounded $\mathcal C$-resolution dimension, where $\mathcal C$ is the class of all $\mathcal S$-filtered modules. In contrast, we use infinite dimensional tilting theory to show that the class of all locally free modules induced by a non-$\sum$-pure-split tilting module is not precovering. Consequently, the class of all locally Baer modules is not precovering for any countable hereditary artin algebra of infinite representation type.      
\end{abstract}

\maketitle

\section*{Introduction} 

A class $\mathcal C$ of modules is said to be \emph{decomposable}, provided there is a cardinal $\kappa$ such that each module in $\mathcal C$ is a direct sum of $< \kappa$-presented modules from $\mathcal C$.  
For example, the class $\mathcal P_0$ of all projective modules is decomposable by a classic theorem by Kaplansky, but the class $\mathcal I _0$ of all injective modules is decomposable only if $R$ is right noetherian, by a classic theorem by Faith and Walker. The decomposability of the class $\rmod R$ of all modules is equivalent to $R$ being a right pure-semisimple ring, so it is quite rare. In fact, the existence of  
a cardinal $\kappa$ such that every direct product of copies of a module $M$ is a direct sum of $< \kappa$-presented modules already implies that the module $M$ is $\sum$-pure-injective, cf.\ \cite[\S4]{HZ}.       

In contrast, deconstructible classes are ubiquitous. Recall \cite{E} that a class $\mathcal C$ is \emph{deconstructible} provided there is a cardinal $\kappa$ such that each module $M \in \mathcal C$ is $\mathcal C ^{<\kappa}$-filtered, where $\mathcal C ^{<\kappa}$ denotes the class of all $< \kappa$-presented modules from $\mathcal C$. Here, for a class $\mathcal D$, a module $M$ is said to be \emph{$\mathcal D$-filtered} (or a a \emph{transfinite extension} of the modules in $\mathcal D$), provided there exists an increasing chain $( M_\alpha \mid \alpha \leq \sigma )$ of submodules of $M$ with the following properties: $M_0 = 0$, $M_\alpha = \bigcup_{\beta < \alpha} M_\beta$ for each limit ordinal $\alpha \leq \sigma$, $M_{\alpha +1}/M_\alpha \cong D_\alpha$ for some $D_\alpha \in \mathcal D$, and $M_\sigma = M$. This chain is called 
a \emph{$\mathcal D$-filtration} of the module $M$.

Clearly, each decomposable class is deconstructible. So is the class $\rmod R$ for any ring $R$, as well as the classes $\mathcal P _n$, $\mathcal I _n$ and $\mathcal F _n$ of all modules of projective, injective, and flat dimension $\leq n$ respectively, see \cite{AEJO} and \cite{BEE}. In fact, if $\mathcal S$ is an arbitrary set of modules, then the class ${}^\perp (\mathcal S ^\perp)$ is deconstructible and closed under transfinite extensions, \cite{ST}. The latter fact implies deconstructibility of many classes of modules studied in homological algebra. (For a class of modules $\mathcal C$, we define its Ext-orthogonal classes by ${}^\perp \mathcal C = \mbox{Ker}\Ext 1R{-}{\mathcal C}$ and $\mathcal C ^\perp = \mbox{Ker}\Ext 1R{\mathcal C}{-}$.)

The key property of deconstructible classes closed under transfinite extensions is that they are precovering, and hence provide for approximations. Recall that a class $\mathcal C$ is \emph{precovering} in case for each module $M$ there exists a morphism $f \in \Hom RCM$ with $C \in \mathcal C$, such that each morphism $f^\prime \in \Hom R{C^\prime}M$ with $C^\prime \in \mathcal C$ factorizes through $f$. Such $f$ is called a \emph{$\mathcal C$-precover} of the module $M$. Given a precovering class $\mathcal C$, it is possible to develop relative homological algebra by replacing the class of all projective modules $\mathcal P _0$ with $\mathcal C$, \cite{EJ}. The abundance of precovering classes also makes it possible to study problems in module theory by choosing the approximations best fitting the particular setting of the problem, cf.\ \cite[Vol.1]{GT}.   

\medskip 
It may appear that all classes of modules closed under transfinite extensions are deconstructible, and hence precovering. The big surprise due to Eklof and Shelah \cite{ES} says that it is consistent with ZFC that the class ${}^\perp \{ \mathbb Z \}$ of all Whitehead groups is not precovering. However, the latter fact is not provable in ZFC, because it is also consistent that $^\perp \{ \mathbb Z \} = \mathcal P _0$, \cite{S}. Further consistency results on the non-deconstructibility of the classes of the form $^\perp \mathcal C$ were proved in \cite{EST}, but it is still an open problem whether there exists (in ZFC) a non-deconstructible class of modules of the form $^\perp \mathcal C$. 

Recently, it has been shown in \cite{HT} that there do exists non-deconstructible classes of modules closed under transfinite extensions (but not of the form $^\perp \mathcal C$), namely the class of all flat Mittag-Leffler modules over any non-right perfect ring. A question arose of whether the latter result is exceptional, or represents a more general phenomenon.   

\medskip
In Section 1 of this paper, we prove further positive results concerning deconstructibility and existence of approximations. We generalize the results from \cite{AEJO} mentioned above in a different direction: we prove that if $\mathcal S$ is a set of modules and $\mathcal C$ the class of all $\mathcal S$-filtered modules, then for each $n \geq 0$, the class of all modules of $\mathcal C$-resolution dimension $\leq n$ is deconstructible, and hence precovering (Corollary \ref{maincor}). While the proof in \cite{AEJO} relies on zigzagging a fixed free resolution, our proof has to take into consideration the extra dimension of the general problem, by simultaneously modifying the initial $\mathcal S$-filtrations of all the terms in the resolution. The key tool making these modifications  possible is the Hill Lemma concerning filtrations of modules.  

In Section 2, we deal with the phenomenon of non-deconstructibility, and non-precovering, for classes closed under transfinite extensions. We present a different proof of the main result from \cite{HT} using trees on cardinals and their decoration by Bass modules. The point is that this proof works in the general setting of classes of locally $\mathcal F$-free modules. We thus show that there are many other instances of non-precovering, and hence non-deconstructible classes in ZFC, coming from the existence of non-$\sum$-pure-split tilting modules $T$ (Theorem \ref{tiltprec}). The case of flat Mittag-Leffler modules studied recently in \cite{BS}, \cite{HT} and \cite{SaT} is just the particular instance of $T = R$ where $R$ is a non-right perfect ring. But the phenomenon spans much further, to all countable hereditary artin algebras $A$ of infinite representation type: we show that the class of all locally Baer $A$-modules is not precovering (Corollary \ref{locbaernprec}).         

\section*{Preliminaries} 

In what follows, $R$ denotes an associative ring with $1$, and $\rmod R$ the category of all (right $R$-) modules. We will use the notation $\rfmod R$ to denote the class of all \emph{strongly finitely presented} modules, i.e., the modules possessing a projective resolution consisting of finitely presented projective modules. 

For a class of modules $\mathcal C \subseteq \rmod R$, we define the infinite Ext-orthogonal classes by ${}^{\perp_\infty} \mathcal C = \bigcap_{1 \leq i}\mbox{Ker}\Ext iR{-}{\mathcal C}$ and $\mathcal C ^{\perp_\infty} = \bigcap_{1 \leq i}\mbox{Ker}\Ext iR{\mathcal C}{-}$, and the Tor-orthogonal class by $\mathcal C ^\intercal = \mbox{Ker}\Tor R1{\mathcal C}{-}$. Similarly, ${}^\intercal \mathcal D = \mbox{Ker}\Tor R1{-}{\mathcal D}$ for each class of left $R$-modules $\mathcal D$. A pair $(\mathcal A, \mathcal B)$ of classes of modules is a \emph{cotorsion pair}, provided that $\mathcal B = \mathcal A ^\perp$ and $\mathcal A = {}^\perp \mathcal B$.        

A module $T$ is \emph{tilting}, provided $T$ has finite projective dimension, $\Ext iRT{T^{(I)}} = 0$ 
for each $i \geq 1$ and each set $I$, and there exist a $k < \omega$ and an exact sequence $0 \to R \to T_0 \to \dots \to T_k \to 0$ such that $T_i \in \Add T$ for each $i \leq k$ (here, $\Add T$ denotes the class of all direct summands of direct sums of copies of $T$). Each tilting module induces a \emph{tilting cotorsion pair} $(\mathcal A,\mathcal B)$ where $\mathcal B = T^{\perp_\infty}$. The class $\mathcal B$ is called the \emph{tilting class} induced by $T$. 
By \cite[13.46]{GT}, each tilting class $\mathcal B$ is of finite type, that is, $\mathcal B = \mathcal S ^\perp$ for $\mathcal S = \mathcal A \cap \rfmod R$. 

There is another cotorsion pair associated with $T$, namely $(\bar{\mathcal A}, \bar{\mathcal B})$ where $\bar{\mathcal A} = {}^\intercal(\mathcal S ^\intercal) = \varinjlim \mathcal S$ is the closure of $\mathcal S$ under direct limits. This cotorsion pair is called the \emph{closure} of $(\mathcal A,\mathcal B)$. The two cotorsion pairs coincide, iff $T$ is \emph{$\sum$-pure-split}, that is, each pure embedding $T_1 \hookrightarrow T_2$ with $T_1, T_2 \in \Add T$ splits, \cite[13.55]{GT}. 

A precovering class of modules $\mathcal C$ is called \emph{special precovering} provided that each module $M$ has a $\mathcal C$-precover $f : C \to M$ which is surjective and satisfies $\mbox{Ker} (f) \in {}^\perp \mathcal C$. Moreover, $\mathcal C$ is called \emph{covering} provided that each module $M$ has a $\mathcal C$-precover $f : C \to M$ with the following minimality property:  $g$ is an automorphism of $C$, whenever $g : C \to C$ is an endomorphism of $C$ with $fg = f$. Such $f$ is called a \emph{$\mathcal C$-cover} of $M$. Dually, we define the notions of a \emph{preenveloping}, \emph{special preenveloping}, and \emph{enveloping} class of modules. 

We note that the class $\mathcal A$ above is special precovering, $\bar{\mathcal A}$ is covering, $\mathcal B$ special preenveloping, and $\bar{\mathcal B}$ enveloping, cf.\ \cite{EJ} or \cite{GT}.

For example, each projective generator $T$ is tilting, and $T$ is $\sum$-pure-split, iff the ring $R$ is right perfect, by a classic theorem by Bass \cite[28.4]{AF}. The two associated cotorsion pairs here are $(\mathcal P _0,\rmod R)$ and $(\mathcal F_0, \mathcal E)$, where $\mathcal E$ is the class of all Enochs cotorsion modules. 

\medskip
For a class of modules $\mathcal C$, we will denote by $\Filt (\mathcal C)$ the class of all $\mathcal C$-filtered modules. The key fact about this class is

\begin{lemma}\label{filt} Let $\mathcal S$ be a set of modules. Then $\Filt (\mathcal S)$ is a precovering class. 
\end{lemma}
\begin{proof} This has been proved in \cite[2.15]{SaS} (see also \cite[7.21]{GT}). 
\end{proof}

Assume that $M \in \Filt (\mathcal C)$, and $\mathcal C$ is a class of $< \kappa$-presented modules for an infinite regular cardinal $\kappa$. 
Let $\mathcal M = ( M_\alpha \mid \alpha \leq \sigma )$ be a $\mathcal C$-filtration of $M$. Then $\mathcal M$ can be expanded into a family 
of submodules of $M$ with the following remarkable properties:
 
\begin{lemma}\label{hill}(Hill Lemma) There is a family $\mathcal H$ consisting of submodules of $M$ such that
\begin{itemize}
\item[(H1)] $\mathcal M \subseteq \mathcal H$;
\item[(H2)] $\mathcal H$ is a complete distributive sublattice of the modular lattice of all submodules of $M$;
\item[(H3)] If $N,P \in \mathcal H$ are such that $N \subseteq P$, then the module $P/N$ is $\mathcal C$-filtered; 
\item[(H4)] Let $N \in \mathcal H$ and $X$ be a subset of $M$ of cardinality $<\kappa$.  
Then there is a $P \in \mathcal H$ such that $N \cup X \subseteq P$ and $P/N$ is $<\kappa$-presented.
\end{itemize}
\end{lemma}
\begin{proof} Consider a family of $<\kappa$-generated modules $(A_\alpha \mid \alpha < \sigma)$, such that for each $\alpha <
\sigma$, we have $ M_{\alpha + 1} = M_\alpha + A_\alpha$. We call a subset $S$ of $\sigma$ is \emph{closed},
provided that each $\alpha \in S$ satisfies $M_\alpha \cap A_\alpha \subseteq \sum_{\beta \in S, \beta<\alpha} A_\beta$.

Let $\mathcal H = \{ \sum_{\alpha \in S} A_\alpha \mid S \hbox{ a closed subset of } \sigma\}$. Then 
$\mathcal H$ satisfies (H1)-(H4) by \cite[7.21]{GT}. 
\end{proof}

For more details and further properties of the notions defined above, we refer to \cite{EJ} and \cite{GT}.  
 
\section{Deconstructibility for $\mbox{Filt}(\mathcal S )$-resolved modules}

By \cite[4.1]{AEJO} (see also \cite[II.3.2]{RG}), the class $\mathcal P _n$ is deconstructible for each $n \geq 0$. Here $\mathcal P _n$ denotes the class of all modules of projective dimension $\leq n$ (= the modules of $\mathcal P _0$-resolution dimension $\leq n$). 

The aim of this section is to establish a more general result which replaces $\mathcal P _0$ by an arbitrary deconstructible class of modules. In fact, Theorem \ref{filt-deconstr} below even shows that the resolutions need not have finite length, and their elements need not belong to the same class.

\begin{definition}
Let $\mathcal C$ be a class of modules and $M$ a module. A long exact sequence
$$ \cdots \longrightarrow C_{i+1} \longrightarrow C_i \longrightarrow \cdots \longrightarrow C_1 \longrightarrow C_0 \longrightarrow M \longrightarrow 0$$
with $C_i \in \mathcal C$ for each $i<\omega$ is called a \emph{$\mathcal C$-resolution} of $M$.

In analogy to the projective case, a module $M$ is said to have \emph{$\mathcal C$-resolution dimension $\leq n$}, provided it possesses a $\mathcal C$-resolution of length $\leq n$.
\end{definition}

\begin{definition}
Let $R$ be a ring and $\kappa$ a cardinal. Then $R$ is called \emph{right $\kappa$-noetherian}, provided that each right ideal $I$ of $R$ is $\leq \kappa$-generated.
The least infinite cardinal $\kappa$ such that $R$ is right $\kappa$-noetherian is the \emph{right dimension of $R$}, denoted by $\dim(R)$. 

For example, if $R$ is right noetherian, then $\dim(R) = \aleph_0$.
\end{definition}

The following Lemma is well-known (see e.g.\ \cite[6.31]{GT}):

\begin{lemma}
\label{dimenze}
Let $\kappa$ be a cardinal such that $\kappa \geq \dim(R)$. Then each submodule of a $\leq \kappa$-generated module is $\leq \kappa$-generated.
In particular, each $\leq\kappa$-generated module is $\leq\kappa$-presented.
\end{lemma}

We can now prove the main result of this section:

\begin{theorem}
\label{filt-deconstr}
Let $\kappa$ be a cardinal such that $\kappa \geq \dim(R)$, $M$ a module and $\mathcal S_1, \mathcal S_2, \dots$ sets of \leqk-presented modules.
Assume that there is a long exact sequence
$$ \cdots \sipka{f_{n+1}} D_n \sipka{f_n} \cdots \sipka{f_2} D_1 \sipka{f_1} D_0 \sipka{f_0} M \longrightarrow 0, \leqno \quad\mathcal R: $$
with $D_i \in \Filt(\mathcal S_i)$ for all $i<\omega$.
Then there is a filtration $(M_\alpha \mid \alpha \leq \lambda)$ of $M$ such that for every $\alpha < \lambda$, there is a long exact sequence
$$ \cdots \sipka{\ol f_{\alpha,n+1}} \ol D_{\alpha,n} \sipka{\ol f_{\alpha,n}} \cdots \sipka{\ol f_{\alpha,2}} \ol D_{\alpha,1} \sipka{\ol f_{\alpha,1}} \ol D_{\alpha,0} \sipka{\ol f_{\alpha,0}} M_{\alpha+1}/M_\alpha \longrightarrow 0 \leqno \quad \ol{\mathcal R}_\alpha : $$
with $\ol D_{\alpha,n} \in \Filt(\mathcal S_n)^{\leq\kappa}$ and $M_{\alpha+1}/M_\alpha$ \leqk-presented.
\end{theorem}

\begin{proof}
\def\lD#1#2{{\displaystyle D_{#1,#2}^\leftarrow}}
\def\rD#1#2{{\displaystyle D_{#1,#2}^\rightarrow}}
\def\lrD#1#2{{\displaystyle D_{#1,#2}^\leftrightarrow}}
\def\Ker{\operatorname{Ker}}
Since $\kappa \geq \dim(R)$, Lemma \ref{dimenze} implies that the notions of a \leqk-presented and a \leqk-generated module coincide.

Let $\lambda = \kappa + \varrho$, where $\varrho$ is the minimal number of generators of $M$, and let $\{m_\alpha \mid \alpha<\lambda\}$ be a generating set of $M$. We will inductively construct a continuous chain of long exact sequences $({\mathcal R}_\alpha \mid \alpha\leq\lambda)$ of the form
$$ \cdots \sipka{f_{n+1}\hrp D_{\alpha,n+1}} D_{\alpha,n} \sipka{f_n \hrp D_{\alpha,n}} \cdots \sipka{f_2 \hrp D_{\alpha,2}} D_{\alpha,1} \sipka{f_1\hrp D_{\alpha,1}} D_{\alpha,0} \sipka{f_0\hrp D_{\alpha,0}} M_\alpha \longrightarrow 0 \leqno \quad {\mathcal R}_\alpha: $$
with $D_{\alpha,n} \in \Filt(\mathcal S_n)$ and ${\mathcal R}_{\alpha+1}/{\mathcal R}_{\alpha} = \ol{\mathcal R}_{\alpha}$.

Denote by $\mathcal H_i$ the family of submodules of $D_i$ obtained from an $\mathcal S_i$-filtration of $D_i$ using the Hill Lemma \ref{hill}; we shall pick the elements of chains $(D_{\alpha,i} \mid \alpha \leq \lambda)$ from these families.

Put $M_0 = 0$ and $D_{0,i} = 0$ for every $i<\omega$ as well. Assume that $M_\alpha$ and $\mathcal R_\alpha$ are already constructed, $D_{\alpha,i} \in \mathcal H_i$ for each $i<\omega$ and $M_\alpha \neq M$. Let $\gamma < \lambda$ be the least index such that $m_\gamma \notin M_\alpha$ (this ensures that $M = \bigcup_{\alpha<\lambda} M_\alpha$).

To begin the construction of $\mathcal R_{\alpha+1}$, choose a $d \in D_0$ with $f_0(d) = m_\gamma$. Property (H4) from Lemma \ref{hill} gives a module $\lD 01 \in \mathcal H_0$ such that $\lD 01 \supseteq D_{\alpha,0} \cup \{d\}$ and $\lD 01 / D_{\alpha,0}$ is \leqk-presented.

Since
$$ \Ker( f_0 \hrp D_{\alpha,0} ) = \Ker( f_0 \hrp \lD 01 ) \cap D_{\alpha,0}, $$
we see that
$$
\Ker( f_0 \hrp \lD 01 ) / \Ker( f_0 \hrp D_{\alpha,0} ) \cong
( \Ker( f_0 \hrp \lD 01 ) + D_{\alpha,0} ) / D_{\alpha,0} \subseteq \lD 01 / D_{\alpha,0}.
$$
As the last module is \leqk-generated, the middle one (and therefore the first one) is \leqk-generated as well by Lemma \ref{dimenze}. Let $G \subseteq \lD 01$ be a set satisfying $|G| \leq \kappa$ and $\Ker( f_0 \hrp D_{\alpha,0}) + \langle G \rangle = \Ker( f_0 \hrp \lD 01)$. Then, because of exactness of $\mathcal R_\alpha$ at $D_{\alpha,0}$, there is a set $H \subseteq D_1$ such that $|H| \leq \kappa$ and $f_1(H) = G$. Hill's property (H4) yields a module $\lD 11 \in \mathcal H_1$ such that $\lD 11 \supseteq D_{\alpha,1} \cup H$ and $\lD 11 / D_{\alpha,1}$ is \leqk-presented. Clearly $f_1(\lD 11) \supseteq \Ker( f_0 \hrp \lD 01)$.

Further, put $\rD 11 = \lD 11$ and construct $\rD 01 \in \mathcal H_0$ so that $f_1(\lD 11) \subseteq \rD 01$ and $\rD 01 / \lD 01$ is \leqk-presented.

The construction proceeds in a similar way: if the modules $\rD 0n, \dots, \rD nn$ are constructed, we let $\lD 0{n+1} = \rD 0n$, and construct the modules $\lD 1{n+1}, \dots, \lD {n+1}{n+1}$ so that $\rD in \subseteq \lD i{n+1}$ and $f_{i+1}(\lD {i+1}{n+1}) \supseteq \Ker( f_i \hrp \lD i{n+1})$ for $i = 0, 1, \dots, n$ (observe, however, that $\lD {n+1}{n+1}$ has to be constructed from $D_{\alpha,n+1}$, as there is no $\rD {n+1}n$). Next, we put $\rD {n+1}{n+1} = \lD {n+1}{n+1}$, and choose $\rD n{n+1}, \dots,\rD 0{n+1}$ satisfying $f_{i+1}(\rD {i+1}{n+1}) \subseteq \rD i{n+1}$ for all $i = 0, 1, \dots, n$. In all these steps, the factors of the newly constructed modules by their submodules constructed earlier are \leqk-presented.

Let $D_{\alpha+1, i} = \bigcup_{i\leq j<\omega} \lD ij$ and $M_{\alpha+1} = f_0(D_{\alpha+1, 0})$. First observe that these modules, together with the restrictions of the maps $f_i$, form an exact sequence --- the ``$\leftarrow$'' steps of the construction ensure that the kernels are inside the images, whereas the ``$\rightarrow$'' steps take care of the inverse inclusion. Morover, as each module $D_{\alpha+1, i}$ is the union of a chain of modules from $\mathcal H_i$ with consecutive factors \leqk-presented, we conclude that $D_{\alpha+1, i} \in \mathcal H_i$ and $D_{\alpha+1, i} / D_{\alpha, i}$ is \leqk-presented for each $i<\omega$.

To see that $M_{\alpha+1}/M_\alpha$ is \leqk-generated, consider the following diagram with both rows exact:

$$\begin{CD}
0@>>>  D_{\alpha,0}@>>> D_{\alpha+1,0}@>>> {D_{\alpha+1,0}/D_{\alpha,0}}@>>>   0\\
@.     @V{f_0 \restriction D_{\alpha,0}}VV    @V{f_0 \hrp D_{\alpha+1,0}}VV @V{g}VV @.\\
0@>>>  {M_\alpha}@>>> {M_{\alpha+1}}@>{\pi}>>   {M_{\alpha+1}/M_\alpha}@>>>   0\\
\end{CD}$$

The map $g\colon D_{\alpha+1,0}/D_{\alpha,0} \to M_{\alpha+1}/M_\alpha$ is defined by $g(x + D_{\alpha,0}) = f_0(x) + M_\alpha$; $g$ is well-defined, because $x-y \in D_{\alpha,0}$ implies $f_0(x-y) \in M_\alpha$, and hence $f_0(x) + M_\alpha = f_0(y) + M_\alpha$. The diagram is easily checked to be commutative. Since both $f_0 \hrp D_{\alpha+1,0}$ and $\pi$ are epimorphisms, $g$ must be epic as well. The module $M_{\alpha+1}/M_\alpha$ is thus a homomorphic image of a \leqk-generated module $D_{\alpha+1,0}/D_{\alpha,0}$, hence it is \leqk-generated itself.

For a limit ordinal $\alpha \leq \lambda$, we put $D_{\alpha,i} = \bigcup_{\beta<\alpha}D_{\beta,i}$ and $M_\alpha = \bigcup_{\beta<\alpha}M_\beta$ and define the morphisms in $\mathcal R_\alpha$ as the corresponding restrictions. Such construction clearly yields that $\mathcal R_\alpha$ is exact, and by the property (H2), we infer that $D_{\alpha,i} \in \mathcal H_i$ for all $i<\omega$.

Since all the complexes $\mathcal R_\alpha$ are exact, the factor complexes $\ol{\mathcal R}_\alpha = \mathcal R_{\alpha+1}/\mathcal R_\alpha$ are also exact,  whence the complexes $\ol {\mathcal R}_\alpha$ have the desired properties.
\end{proof}

\begin{corollary}\label{maincor}
Let $\mathcal C$ be a deconstructible class of modules.
\begin{enumerate}
\item The class of all modules possessing a $\mathcal C$-resolution is deconstructible.
\item The class of all modules of $\mathcal C$-resolution dimension $\leq n$ is deconstructible for each $n<\omega$; in particular, the classes $\mathcal P_n$ are deconstructible.
\end{enumerate}
\end{corollary}
\begin{proof}
Let $\mathcal S \subseteq \ModR$ be a set such that $\mathcal C = \Filt(\mathcal S)$. The case (1) is obtained by taking $\mathcal S_i = \mathcal S$ in the Theorem \ref{filt-deconstr}, the case (2) by taking $\mathcal S_i = \mathcal S$ for $i \leq n$ and $\mathcal S_i = \{0\}$ for $n < i < \omega$. Since $\mathcal P_0$ is deconstructible (in fact, decomposable), the claim concerning $\mathcal P_n$ is just a special case of (2).
\end{proof}

\section{Locally $\mathcal F$-free modules}

From now on, $\mathcal F$ will denote a class of countably presented modules, and $\mathcal C$ the class of all countably $\mathcal F$-filtered modules. By Lemma \ref{hill}, each module $M \in \mathcal C$ has a $\mathcal C$-filtration of length $\leq \omega$. 

\begin{definition}\label{def_loc_free} Let $M$ be a module. Then $M$ is \emph{locally $\mathcal F$-free}, provided there exists a set $\mathcal S$ consisting of submodules of $M$ such that 
 
\begin{itemize}
\item[\rm{(S1)}] $\mathcal S \subseteq \mathcal C$;
\item[\rm{(S2)}] For each countable subset $C$ of $M$ there exists $S \in \mathcal S$ such that $C \subseteq S$;
\item[\rm{(S3)}] $0 \in \mathcal S$, and $\mathcal S$ is closed under unions of countable chains.
\end{itemize}

The set $\mathcal S$ is said to \emph{witness} the local $\mathcal F$-freeness of $M$. 
We will denote by $\mathcal L$ the class of all locally $\mathcal F$-free modules.
\end{definition}

If we view $\mathcal F$-filtered modules as the {\lq free\rq} ones, and the elements of $\mathcal L$ as the {\lq locally free\rq} modules, then the next lemma just says that {\lq free\rq} implies {\lq locally free\rq}, and the converse holds for countably generated modules:

\begin{lemma}\label{hillappl} Each $\mathcal F$-filtered module is locally $\mathcal F$-free. The class $\mathcal C$ coincides with the class of all countably generated locally $\mathcal F$-free modules. 
\end{lemma}
\begin{proof} Let $\mathcal M = ( M_\alpha \mid \alpha \leq \sigma )$ be an $\mathcal F$-filtration of $M$ and $A_\alpha$ ($\alpha < \sigma$) be countably generated modules such that $M_{\alpha + 1} = M_\alpha + A_\alpha$ for each $\alpha < \sigma$. Let $\mathcal H = \{ \sum_{\alpha \in S} A_\alpha \mid S \mbox{ closed in } \sigma \}$ be the family defined in the proof of Lemma \ref{hill}. We let $\mathcal S = \{ \sum_{\alpha \in S} A_\alpha \mid S \mbox{ countable and closed in } \sigma \}$. Then $\mathcal S$ witnesses the local $\mathcal F$-freeness of $M$. 

If $M$ is countably generated with a witnessing set $\mathcal S$ for local $\mathcal F$-freeness, then $M \in \mathcal C$ by conditions (S1) and (S2). 
\end{proof}

We will denote by $\varinjlim_{\omega} \mathcal F$ the class of all countable direct limits of the modules from $\mathcal F$, and by $\mathcal D$ the class of all direct summands of the modules $M$ that fit into an exact sequence $0 \to P \to M \to C \to 0$ where $P$ is a free module (i.e., $P \cong R^{(I)}$ for a set $I$) and $C \in \mathcal C$.  

\begin{example}\label{MLtilt} (i) If $\mathcal F$ is the class of all countably generated projective modules, then $\mathcal F$-filtered = projective, 
$\varinjlim_{\omega} \mathcal F$ is the class of all countably presented flat modules, and $\mathcal L$ is the class of all flat Mittag-Leffler modules, 
see \cite{HT} or \cite[\S 3.2]{GT}. 

(ii) Let $T$ be a countably generated tilting module and $\mathcal F = \{ T \}$. Then $\mathcal F$-filtered = isomorphic to a direct sum of copies of $T$. 

Similarly, if $T$ is a $\sum$-pure-injective tilting module, and $\mathcal F$ is a representative set of all indecomposable direct summands in a fixed indecomposable decomposition of $T$, then the class of all $\mathcal F$-filtered modules coincides with $\Add T$.

(iii) Let $R$ be a hereditary artin algebra of infinite representation type and $\mathcal F$ be a representative set of all finitely generated preprojective modules. Then the class of all $\mathcal F$-filtered modules coincides with the class of all \emph{Baer modules}, cf.\ \cite{AKT} or \cite[\S 14.3.2]{GT}. In this case, the modules in the class $\mathcal L$ will be called \emph{locally Baer}.   
\end{example}

In general the class $\mathcal L$ is not closed under direct summands: if $\mathcal F = \{ R \}$, where $R$ is the completely reducible ring $R = K \oplus K$ and $K$ is a field, then $\mathcal L$ is the class of all free modules, so $\mathcal L$ does not contain the projective module $K \oplus 0$. 

However, $\mathcal L$ is always closed under transfinite extensions:
 
\begin{theorem}\label{transf_closed} The class $\mathcal L$ is closed under transfinite extensions.
\end{theorem}
\begin{proof} 
Let $M$ be a module possessing an $\mathcal L$-filtration $(M_\alpha \mid \alpha \leq \lambda)$. By induction on $\alpha \leq \lambda$, we will construct the sets $\mathcal S_\alpha$ witnessing the local $\mathcal F$-freeness of $M_\alpha$ so that $\mathcal S_\gamma \subseteq \mathcal S_\delta$ for all  $\gamma\leq\delta\leq\lambda$, and the following condition \eqref{ast} is satisfied:
\begin{equation}
\vcenter{\advance\hsize-4\parindent\noindent If $\gamma < \delta \leq \lambda$ and $S \in \mathcal S _\delta$, then $S \cap M_\gamma \in \mathcal S_ \gamma$ and $S/(S \cap M_\gamma) \in \mathcal C$.} \tag{$\ast$}
\label{ast}
\end{equation}
The set $S_\lambda$ will then witness the $\mathcal F$-local freeness of $M = M_\lambda$. There is nothing to prove for $\alpha \leq 1$. In the inductive step,  
we distinguish two cases, depending on whether $\alpha$ is a successor or a limit ordinal.

\medskip
\noindent\emph{The successor case.} \,
Suppose that we have already constructed $\mathcal S_\alpha$. By assumption, $M_{\alpha +1}/M_\alpha \in \mathcal L$. Let  
$\pi : M_{\alpha +1} \to M_{\alpha +1}/M_\alpha$ be the projection, and $\bar{\mathcal S}_\alpha$ be a set witnessing 
the local $\mathcal F$-freeness of $M_{\alpha +1}/ M_\alpha$. Put
$$\mathcal S_{\alpha +1} = \bigl\{S \subseteq M_{\alpha +1} \mid S \cap M_\alpha \in \mathcal S _\alpha \mathrel\& \pi(S) \in \bar{\mathcal S}_\alpha \bigr\}.$$
Since $0 \in \bar{\mathcal S}_\alpha$, the inclusion $\mathcal S _\alpha \subseteq \mathcal S _{\alpha +1}$ is clear, so by the inductive hypothesis, $\mathcal S _\beta \subseteq \mathcal S _{\alpha +1}$ for all $\beta \leq \alpha +1$. 

Let $S \in \mathcal S _{\alpha +1}$. Then $S \cap M_\alpha \in \mathcal S _\alpha \subseteq \mathcal C$ and $S/(S \cap M_\alpha) \cong (S + M_\alpha)/M_\alpha = \pi(S) \in \bar{\mathcal S}_\alpha \subseteq \mathcal C$. This shows that condition ($\ast$) holds for $\gamma = \alpha$ and $\delta = \alpha +1$. Moreover, (S1) holds for $\mathcal S _{\alpha +1}$, because $\mathcal C$ is closed under extensions. 
Since $S \cap M_\gamma = (S \cap M_\alpha) \cap M_\gamma$ and there is an exact sequence $0 \to (S \cap M_\alpha)/(S \cap M_\gamma) \to S/(S \cap M_\gamma) \to S/(S \cap M_\alpha) \to 0$ for all $\gamma < \alpha$, the validity of condition ($\ast$) for all $\gamma < \delta \leq \alpha + 1$ now follows by the inductive premise.   

In order to prove condition (S2), consider a countable subset $C$ of $M_{\alpha +1}$. There exists $\bar{S} \in \bar{\mathcal S}_\alpha$ such that $\pi (C) \subseteq \bar{S}$. Since $\bar{S} \in \mathcal C$, $\bar{S}$ is countably presented, so there is a countably generated module $T \subseteq M_{\alpha +1}$ such that $\pi(T) = \bar{S}$, $C \subseteq T$, and $T \cap M_\alpha = \Ker{\pi \restriction T}$ is countably generated. Also, there exists an $S \in \mathcal S _\alpha$ satisfying $T \cap M_\alpha \subseteq S$. Then $T \cap S = T \cap M_\alpha$ and $(T + S) \cap M_\alpha = S$, and the exact sequence $0 \to S \hookrightarrow T+S \to \bar{S} \to 0$ yields $T+S \in \mathcal S _{\alpha +1}$. Since $C \subseteq T+S$, we have established (S2).

Finally, consider a chain $S_0 \subseteq \dots \subseteq S_i \subseteq S_{i+1} \subseteq \dots$ of elements of $\mathcal S _{\alpha +1}$. Let $S = \bigcup_{i<\omega} S_i$. Since $S_i \cap M_\alpha \in \mathcal S _\alpha$ for all $i<\omega$ and $S \cap M_\alpha = \bigcup_{i<\omega} (S_i \cap M_\alpha)$, we infer that $S \cap M_\alpha \in \mathcal S _\alpha$. Similarly, $\pi(S) = \bigcup_{i<\omega} \pi(S_i) \in \bar{\mathcal S}_\alpha$ as $\pi(S_i) \in \bar{\mathcal S}_\alpha$ for all $i<\omega$. Thus $S \in \mathcal S _{\alpha +1} $ and condition (S3) holds.

\medskip
\noindent\emph{The limit case.}\,
Let $\alpha \leq \lambda$ be a limit ordinal and assume that the systems $\mathcal S_\beta$ for $\beta<\alpha$ are already constructed.
Put
$$ \mathcal S _\alpha = \bigl\{ {\textstyle\bigcup_{i<\omega} S_i} \mid S_i \in \mathcal S _{\alpha_i} \text{ for some } \alpha_i < \alpha \text{ and }  S_i \subseteq S_{i+1}, \text{ for each } i<\omega  \bigr\}.$$
Then $\mathcal S _{\alpha}$ consists of submodules of $M_\alpha$, and the inclusion $\mathcal S _\beta \subseteq \mathcal S _\alpha$ is clear for each $\beta < \alpha$.

Consider an arbitrary $S = \bigcup_{i<\omega} S_i \in \mathcal S _\alpha$. We may assume that the sequence $(\alpha_i \mid i<\omega)$ is increasing. Let $\beta = \sup_{i < \omega} \alpha_i$ (so $S \subseteq M_\beta = \bigcup_{i<\omega} M_{\alpha_i}$). For each $i < \omega$, let $S_i^\prime = S \cap M_{\alpha_i} = \bigcup_{j<\omega} (S_j \cap M_{\alpha_i})$. Then $S_i^\prime \in \mathcal S_{\alpha_i}$ by condition ($\ast$) for $\gamma = \alpha_i$. Moreover, $S = \bigcup_{i<\omega} S_i^\prime$ and $S_i^\prime \cap M_{\alpha_j} = S_j^\prime$ for all $j \leq i < \omega$. Thus we can also assume that $S_i =  S \cap M_{\alpha_i} \in \mathcal S_{\alpha_i}$ for each $i < \omega$. 

Notice that by the inductive premise, we have $S \cap M_\gamma = \bigcup_{i < \omega} (S_i \cap M_\gamma) \in \mathcal S _\gamma$ for each $\gamma < \alpha$. If $\gamma \geq \beta$, then $S = S \cap M_\gamma$; otherwise, consider an $i < \omega$ such that $\gamma < \alpha_i$. Then $S \cap M_\gamma = (S \cap M_{\alpha_i}) \cap M_\gamma = S_i \cap M_\gamma$. 

Condition (S1) now follows from the fact that $S_{i+1}/S_i =  S_{i+1}/(S_{i+1} \cap M_{\alpha_i}) \in \mathcal C$ for each $i < \omega$, so $S$ is countably $\mathcal C$-filtered, whence $S \in \mathcal C$. 

Moreover, for $\gamma < \alpha_i$, the outer terms of the exact sequence $0 \to S_i/(S \cap M_\gamma) \to S/(S\cap M_\gamma) \to S/S_i \to 0$ belong to $\mathcal C$. Thus $S/(S\cap M_\gamma) \in \mathcal C$, and we infer that condition ($\ast$) holds for all $\gamma < \delta \leq \alpha$.    

In order to prove condition (S2), consider a countable subset $C = \{c_i \mid i<\omega\}$ in $M_\alpha$. Let $\alpha_i < \alpha$ be ordinals such that $c_i \in M_{\alpha_i}$ for each $i < \omega$. Again, we may assume that the sequence $(\alpha_i \mid i<\omega)$ is increasing. By induction on $i < \omega$ we construct a chain of modules $S_0 \subseteq \dots \subseteq S_i \subseteq S_{i+1} \subseteq \dots$ such that $\{c_0, \dots, c_i\} \subseteq S_i$ and $S_i \in \mathcal S_{\alpha_i}$ for all $i<\omega$. Let $S_0 \in \mathcal S _{\alpha_0}$ be such that $c_0 \in S_0$. Assuming we have already constructed $S_0, \dots, S_i$, we choose $S_{i+1} \in \mathcal S _{\alpha_{i+1}}$ so that $D_i \cup \{c_{i+1}\} \subseteq S_{i+1}$, where $D_i \subseteq M _{\alpha_i}\, (\subseteq M_{\alpha_{i+1}})$ is a countable set generating $S_i$. Clearly $S_i \subseteq S_{i+1}$ and $C \subseteq \bigcup_{i<\omega} S_i \in \mathcal S_\alpha$. This proves condition (S2).

For the verification of (S3), let $S_0 \subseteq \dots \subseteq S_i \subseteq S_{i+1} \subseteq \dots$ be a chain of modules from $\mathcal S_\alpha$, where $S_i = \bigcup_{j<\omega} S_{ij}$, $S_{ij} \in \mathcal S_{\alpha_{ij}}$, and as above, we can assume that $(\alpha_{ij} \mid j<\omega)$ is an increasing sequence of ordinals $<\alpha$ for each $i<\omega$, and $S_{ij} = S_i \cap M_{\alpha_j}$ for each $j < \omega$. 

Let $S = \bigcup_{i<\omega} S_i$ and $\alpha ^\prime = \sup_{i,j<\omega} \alpha_{ij}$. If $\alpha ^\prime < \alpha$, then $S_{ij} \in \mathcal S _{\alpha ^\prime}$ for all $i,j < \omega$, whence $S_i \in \mathcal S _{\alpha ^\prime}$ for all $i<\omega$, and $S \in \mathcal S_{\alpha ^\prime} \subseteq \mathcal S_\alpha$. 

If $\alpha ^\prime = \alpha$, then $\alpha$ has cofinality $\omega$. Let $(\beta_k \mid k<\omega)$ be an increasing sequence of ordinals $<\alpha$ with $\sup_{k<\omega}\beta_k = \alpha$. By condition ($\ast$) for $\delta = \alpha$, we have $S_i \cap M_{\beta _k} \in \mathcal S _{\beta _k}$ for each $i < \omega$.
So for each $k < \omega$, $S \cap M_{\beta_k} = \bigcup_{i<\omega}(S_i \cap M_{\beta_k}) \in \mathcal S_{\beta_k}$. We conclude that 
$S = \bigcup_{k<\omega}(S \cap M_{\beta_k}) \in \mathcal S_\alpha$ by the definition of $\mathcal S _\alpha$ above.  
\end{proof}

\section{Trees and Bass modules}

Next we turn to Bass modules and trees on cardinals. They will form the algebraic and combinatorial background, respectively, for our construction of particular locally $\mathcal F$-free modules.

\begin{definition}\label{bass_module} A module $M$ is a \emph{Bass module} for $\mathcal F$, provided that $M$ is the direct limit of a direct system 
\begin{equation}\label{e0}
F_0 \overset{g_0}\to F_1 \overset{g_1}\to \dots \overset{g_{i-1}}\to F_i \overset{g_i}\to F_{i+1} \overset{g_{i+1}}\to \dots ,
\end{equation}
where $F_i \in \mathcal F$ and $g_i \in \Hom R{F_i}{F_{i+1}}$ for all $i < \omega$. 

Clearly, $M \in \varinjlim_{\omega} \mathcal F$, and the canonical presentation of the direct limit yields 
\begin{equation}\label{e1}
0 \to \bigoplus_{i < \omega} F_i \overset{f}\to \bigoplus_{i < \omega} F_i \to M \to 0,
\end{equation}
where $f(x) = x - g_i(x)$ for all $i < \omega$ and $x \in F_i$.      
\end{definition} 

Conversely, each module $M \in \varinjlim_{\omega} \mathcal F$ is of the form (\ref{e1}), and hence it is a Bass module for $\mathcal F$, see \cite[2.12]{GT}. 
So $\varinjlim_{\omega} \mathcal F$ coincides with the class of all Bass modules for $\mathcal F$. 

\begin{example}\label{Bass} Our terminology comes from the fact that if $R$ is a non-right perfect ring with a strictly decreasing chain of principal left ideals $( Ra_i...a_0 \mid i < \omega )$, $\mathcal F = \{ R \}$, and $g_i : R \to R$ is the left multiplication by $a_i$, then $M$ is the countably presented flat module used by Bass to prove his Theorem~P, see \cite[\S28]{AF}.
\end{example} 

\medskip
We now define a combinatorial pattern for constructing locally $\mathcal F$-free modules. 

\begin{definition}\label{trees} 
Let $\kappa$ be an infinite cardinal, and $T_\kappa$ be the set of all finite sequences of ordinals $< \kappa$, so 
$T_\kappa = \{ \tau : n \to \kappa \mid n <\omega \}.$

Partially ordered by inclusion, $T_\kappa$ is a tree, called the \emph{tree on $\kappa$}. Notice that $\mbox{card}(T_\kappa) = \kappa$.
For each $\tau \in T_\kappa$, we will denote by $\ell(\tau)$ the length of $\tau$ .

Let $\mbox{Br}(T_\kappa)$ denote the set of all branches of $T_\kappa$. Each $\nu \in \mbox{Br}(T_\kappa)$ can be identified with an $\omega$-sequence of ordinals $< \kappa$, so $\mbox{Br}(T_\kappa) = \kappa^\omega$.
\end{definition}

Our construction of locally $\mathcal F$-free modules consists in decorating the trees $T_\kappa$ with the Bass modules $M \in \varinjlim_{\omega} \mathcal F$ using the direct system (\ref{e0}).
 
\begin{definition}\label{decorated}
Let $D = \bigoplus_{\tau \in T_\kappa} F_{\ell(\tau)}$, and $P = \prod_{\tau \in T_\kappa} F_{\ell(\tau)}$. 
Moreover, let $L$ be a module $D \subseteq L \subseteq P$ defined as follows:

For $\nu \in \mbox{Br}(T_\kappa)$, $i < \omega$, and $x \in F_i$, we take $x_{\nu i} \in P$ such that 
\begin{itemize}
\item $\pi_{\nu \restriction i} (x_{\nu i}) = x$, 
\item $\pi_{\nu \restriction j} (x_{\nu i}) = g_{j-1}\dots g_i(x)$ for all $i < j < \omega$, and 
\item $\pi_\tau (x_{\nu i}) = 0$ otherwise,
\end{itemize}
where $\pi_\tau \in \Hom {R}{P}{F_{\ell(\tau)}}$ denotes the $\tau$th projection for each $\tau \in T_\kappa$.

Let $Y_{\nu i} = \{ x_{\nu i} \mid x \in F_i \}$. Then $Y_{\nu i}$ is a submodule of $P$ isomorphic to $F_i$ via the assignment $x \mapsto x_{\nu i}$. 

Put $X_{\nu i} = \sum_{j \leq i} Y_{\nu j}$. Then $X_{\nu i} \subseteq X_{\nu, i+1}$, and
$X_{\nu i} = \bigoplus_{j < i} F_j \oplus Y_{\nu i} \cong \bigoplus_{j \leq i} F_j$.

Finally, we define $X_\nu = \bigcup_{i < \omega} X_{\nu i}$, and $L = \sum_{\nu \in \mbox{Br}(T_\kappa)} X_\nu$. 
\end{definition}

Next, we present the basic properties of the modules $X_\nu$ and $L$:

\begin{lemma} 
$X_\nu \cong \bigoplus_{i < \omega} F_i$ for each $\nu \in \mbox{Br}(T_\kappa)$.
\end{lemma}
\begin{proof} The inclusion $X_{\nu i} \subseteq X_{\nu, i+1}$ splits, since there is a split exact sequence 
$$0 \to Y_{\nu i} \overset{p}\hookrightarrow F_i \oplus Y_{\nu, i+1} \overset{q}\to F_{i+1} \to 0$$
where $p (x_{\nu i}) = x + (g_i(x))_{\nu, i+1}$, and $q (x + y_{\nu, i+1}) = y - g_i(x)$.  
\end{proof}
 
\begin{lemma}\label{L} $L/D \cong N^{(\mbox{Br}(T_\kappa))}$, and $L$ is locally $\mathcal F$-free.
\end{lemma}
\begin{proof} Let $\nu \in  \mbox{Br}(T_\kappa)$. Then $N \cong (X_\nu + D)/D$. Indeed,  for each $i < \omega$, we can define $f_i : F_i \to (X_\nu + D)/D$ by $f_i(x) = x_{\nu i} + D$. Then $((X_\nu + D)/D, f_i \mid i \in I)$ is the direct limit of the direct system (\ref{e0}).

Since each element of $X_\nu$ is a sequence in $P$ whose $\tau$th component is zero for all $\tau \notin \{ \nu \restriction i \mid i < \omega \}$, the modules $((X_\nu + D)/D \mid \nu \in \mbox{Br}(T_\kappa) )$ are independent. Thus $L/D = \bigoplus_{\nu \in  \mbox{Br}(T_\kappa)} (X_\nu + D)/D \cong N^{(\mbox{Br}(T_\kappa))}$.

For each countable subset $C = \{ \nu_i \mid i < \omega \}$ of $\mbox{Br}(T_\kappa)$, the module $X_C = \sum_{\nu \in C} X_\nu$ 
is isomorphic to a countable direct sum of the $F_i$s. Indeed, $X_C = \bigcup_{i < \omega} X_{C_i}$, where 
$X_{C_i} = \sum_{j \leq i} X_{\nu_j}$ is a direct summand in $X_{C_{i+1}}$, with the complementing summand 
isomorphic to a countable direct sum of the $F_i$s. So the local $\mathcal F$-freeness of $L$ is witnessed by the set 
$\mathcal S$ of all $X_C$, where $C$ runs over all countable subsets of $\mbox{Br}(T_\kappa)$. 
\end{proof}

\section{The non-deconstructibility of locally $\mathcal F$-free modules}

We are going to apply the locally $\mathcal F$-free modules constructed above by the decoration of trees with the Bass modules. We use them to prove the non-deconstructibi\-lity of the class of all locally $\mathcal F$-free modules in the case when $\varinjlim_{\omega} \mathcal F \nsubseteq \mathcal D$. 

The point is that in our setting $\mathcal L ^\perp \subseteq (\varinjlim_{\omega} \mathcal F)^\perp$. The 
proof of this fact is by a simple counting argument using almost no algebra: 

\begin{lemma}\label{griffith_counting} Let $M$ be a module such that $\Ext 1RLM = 0$ for each locally $\mathcal F$-free module $L$. 
Then $\Ext 1RNM = 0$ for each module $N \in \varinjlim_{\omega} \mathcal F$.
\end{lemma} 
\begin{proof} Let $\kappa$ be an infinite cardinal such that $\kappa^\omega = 2^\kappa$ and $M$ has cardinality $\leq 2^\kappa$ 
(such cardinal exists e.g.\ by \cite[8.26(a)]{GT}). Consider $N \in \varinjlim_{\omega} \mathcal F$. Then $N$ is a Bass module with a direct limit presentation \ref{e0}.

Let $L$ be the corresponding locally $\mathcal F$-free module fitting into the exact sequence $0 \to D \to L \to N^{(\hbox{Br}(T_\kappa))} \to 0$ (see Lemma \ref{L}). Applying the functor $\Hom R{-}M$ to this sequence, we see that the connecting homomorphism $\Hom RDM \to \Ext 1R{N^{(\hbox{Br}(T_\kappa))}}M$ is surjective.

Assume that $\Ext 1RNM \neq 0$. Since $\kappa ^\omega = 2^\kappa$, 
$\hbox{card}(\Ext 1R{N^{(\hbox{Br}(T_\kappa))}}M)  = \hbox{card}(\Ext 1RNM ^{\kappa ^\omega}) \geq 2^{2^\kappa}$, while $\hbox{card}(\Hom RDM )  = \hbox{card}(M^\kappa) \leq 2^\kappa$, a contradiction. 
\end{proof}

Now, we can present our main result concerning deconstructibility:

\begin{theorem}\label{non-dec} Assume there exists a module $C \in ( \varinjlim_{\omega} \mathcal F ) \setminus \mathcal D$. 
Then $\mathcal L \neq {}^\perp \mathcal E$ for any class of modules $\mathcal E$, and $\mathcal L$ is not deconstructible. 
\end{theorem}   
\begin{proof} First, assume that $\mathcal L = {}^\perp \mathcal E$ for a class $\mathcal E$. Then $\mathcal E \subseteq \mathcal L ^\perp \subseteq (\varinjlim_{\omega} \mathcal F)^\perp$ by Lemma \ref{griffith_counting}, hence 
$C \in \varinjlim_{\omega} \mathcal F \subseteq {}^\perp \mathcal E = \mathcal L$. Since $C$ is countably presented, $C \in \mathcal C \subseteq \mathcal D$ by Lemma \ref{hillappl}, a contradiction. 

Assume there is a cardinal $\kappa$ such that $\mathcal L = \Filt (\mathcal L ^{< \kappa})$.  
Consider the cotorsion pair $(^\perp((\mathcal L ^{< \kappa})^\perp), (\mathcal L ^{< \kappa})^\perp)$. 
By the Eklof Lemma \cite[6.2]{GT}, $(\mathcal L ^{< \kappa})^\perp = \mathcal L ^\perp$, so $C \in {}^\perp((\mathcal L ^{< \kappa})^\perp)$ by Lemma \ref{griffith_counting}.
By \cite[6.13]{GT}, $C$ is isomorphic to a direct summand in a module $E$ of the form $0 \to P \to E \overset{\pi}\to L \to 0$, where $P$ is a free module and $L \in \mathcal L$. Since $C$ is countably generated, condition (S2) for $L$ implies that $\pi(C)$ is contained in some $D \in \mathcal C$, whence $C$ is a direct summand in $\pi^{-1}(D)$. Thus $C \in \mathcal D$, a contradiction.
\end{proof} 

We record a particular instance of Theorem \ref{non-dec} in the tilting setting:

\begin{corollary}\label{tilt} Let $T$ be a tilting module which is a direct sum of countably presented modules, $T = \bigoplus_{i \in I} T_i$. Let $\mathcal F$ be a representative set of $\{ T_i \mid i \in I \}$. Assume there exists a module $C \in ( \varinjlim_{\omega} \mathcal F ) \setminus \mathcal D$. Then $\mathcal L$ is not deconstructible.     
\end{corollary}

Specializing further, we recover \cite[7.3]{HT}: 

\begin{corollary}\label{ML} 
Let $R$ be a non-right perfect ring. Then the class of all flat Mittag-Leffler modules is not deconstructible.  
\end{corollary}
\begin{proof} Let $\mathcal F$ be a representative set of all countably presented projective modules and $T = \bigoplus_{P \in \mathcal F} P$. 
Clearly, $T$ is a tilting module. Moreover, $\mathcal D = \mathcal C$ is the class of all countably presented projective modules, and $\mathcal L$ is the class of all flat Mittag-Leffler modules (see \cite[2.10]{HT} or \cite[3.19]{GT}). Since $R$ is not right perfect, there exists a countably presented flat module $C$ which is not projective. Then $C \in \varinjlim_{\omega} \mathcal F \setminus \mathcal D$ and Theorem \ref{non-dec} applies.
\end{proof}

 Next, we consider an application to Dedekind domains:

\begin{corollary}\label{Dedekind} Let $R$ be a Dedekind domain with the quotient field $Q \neq R$, and $P$ be a non-empty set of maximal ideals in $R$, such that $\mspec R \setminus P$ is countable. Let $\mathcal F = \{ \bigcap_{p \in P} R_{(p)} \} \cup \{ E(R/q) \mid q \in \mspec R \setminus P \}$. Then the class $\mathcal L$ is not deconstructible.
\end{corollary}
\begin{proof} Let $R_P = \bigcap_{p \in P} R_{(p)}$. Then $R_P$ is a subring of $Q$ containing $R$, so $R_P$ is a Pr\" ufer domain by \cite[III.1.1(d)]{FS}, and hence a Dedekind domain by \cite[11.7]{M}. Take $p \in P$ and $0 \neq x \in p \setminus p^2$. Denote by $C$ the Bass $R_P$-module corresponding to the choice of $a_i = x$ for all $i < \omega$ (see Example \ref{Bass}). Then $C$ is isomorphic to an $R_P$-submodule of $Q$, and $C \in \varinjlim_{\omega} \mathcal F$ (both as an $R_P$-module, and an $R$-module).  

Let $T = R_P \oplus \bigoplus_{p \neq q \in \mspec R} E(R/q)$. Then $T$ is a tilting module, and $T$ induces a cotorsion pair $(\mathcal A, \mathcal B)$ such that $\mathcal D \subseteq \mathcal A$ and $\mathcal B = \{ M \in \rmod R \mid Mq = 0 \mbox{ for all } q \in \mspec R \setminus P \}$ (see \cite[14.30]{GT}).

From the exact sequence $0 \to R \to R_P \to \bigoplus_{p \in \mspec R \setminus P} E(R/p) \to 0$ and the assumption on $P$, we infer that $R_P$ is a countably generated module. The claim will thus follow from Theorem \ref{tilt} once we prove that $C \notin \mathcal D$. 

Assume $C \in \mathcal D$. Then $C \in \mathcal A \cap \mathcal B = \Add T$, and since $C$ is a torsion-free module, $C$ is isomorphic to a direct summand in a direct sum of copies of $R_P$. However, $Cx = C$, while $R_P$ contains no non-zero submodule $N$ such that $Nx = N$, because $\bigcap_{n < \omega} x^n R_P = 0$ by the Krull Intersection Theorem. 
\end{proof}

\begin{remark}\label{except} (i) In the case of Corollary \ref{Dedekind}, a tilting module is obtained also for $P = \emptyset$, namely 
$T = Q \oplus Q/R$. However, if $\mathcal F = \{ Q \} \cup \{ E(R/q) \mid q \in \mspec R \}$, then $\mathcal L$ is the class of all divisible (= injective) modules, which is deconstructible, since $R$ is noetherian (In fact, $\mathcal L$ is even decomposable in this case, by the theorem of Faith and Walker). 
Also note that for $P = \mspec R$, $\mathcal L$ is the class of all modules $M$ such that each countably generated submodule of $M$ is free.    

(ii) More in general, the tilting module $T$ in Theorem \ref{tilt} cannot be $\sum$-pure-split (and in particular, it cannot be $\sum$-pure-injective), cf.\ \cite[13.55]{GT}: Otherwise, since the presentation of $C$ as an element of $\varinjlim_{\omega} \mathcal F$ has the form of a pure-exact sequence
\begin{equation}\label{e2}
0 \to U \to V \to C \to 0
\end{equation} 
where $U$ and $V$ are some countable direct sums of the modules $T_i$, (\ref{e2}) splits, whence $C$ is a direct summand in a countable direct sum of copies of the $T_i$s. Then $C \in \mathcal D$, a contradiction. However, as we will see in the next section, non-$\sum$-pure-split tilting modules provide for a source of non-precovering, and hence non-deconstructible classes even in the setting of artin algebras.    
\end{remark} 

There do exist non-deconstructible classes even in the setting of perfect rings. Our first example of this phenomenon employs the Lukas tilting module $L$ 
over a tame hereditary algebra (see \cite{AHT} or \cite[Example 13.7]{GT}):

\begin{corollary}\label{Lukas} Let $R$ be a finite dimensional tame hereditary algebra, such that the generic module $G$ is countably generated (e.g., the Kronecker algebra over a countable algebraically closed field). Let $\mathcal F = \{ L^{(\omega)} \}$ where $L$ is the Lukas tilting module. Then the class $\mathcal L$ is not deconstructible.  
\end{corollary}
\begin{proof} By \cite[Proposition 7]{AHT}, there are two cotorsion pairs, $(\mathcal T,\mathcal E)$, and $(\mathcal B,\mathcal L)$, in $\rmod R$. The first one is cotilting and generated by $G$, the second one is tilting and generated by $L$, and $\mathcal B \subsetneq \mathcal T$. The modules in $\mathcal T$ are called \emph{torsion-free}, while the ones in $\mathcal B$ are the Baer modules from Example \ref{MLtilt}(iii), cf.\ \cite[\S14.3.2]{GT}. By \cite[Corollary 11]{AHT}, there is a pure-exact sequence $0 \to A \to B \to G \to 0$ such that $B$ is Baer and $A \in \Add L$. Since $G \in \mathcal E \subseteq \mathcal L$, also $B \in \Add L = \mathcal B \cap \mathcal L$, so by Eilenberg's trick, there is a pure exact sequence $0 \to L^{(\kappa)} \overset{g}\to L^{(\lambda)} \to G \to 0$ for some cardinals $\kappa, \lambda$. Since both $L$ and $G$ are countably generated, there exist countable subsets $E \subseteq \kappa$ and $F \subseteq \lambda$ such that $L^{(E)} = L^{(F)} \cap L^{(\kappa)}$ and $L^{(\lambda)} = L^{(F)} + L^{(\kappa)}$, so we can further assume that $\kappa = \lambda = \omega$. 

Let $X = L^{(\omega)}$. Then there is a Bass module $H$ fitting into the pure exact sequence 
\begin{equation}\label{dagger}
0 \to X^{(\omega)} \overset{h}\to X^{(\omega)} \to H \to 0
\end{equation} 
where $h$ maps $x$ from the $i$th copy of $X$ to $x - g(x)$, and $g(x)$ is taken in the $(i+1)$th copy of $X$, for each $i < \omega$. The purity of the sequence implies that $H$ is torsion-free, and clearly $H \in \varinjlim_{\omega} \mathcal F$. 

It remains to prove that $H$ is not Baer (then $H \in ( \varinjlim_{\omega} \mathcal F ) \setminus \mathcal D$, since $\mathcal D$ consists of Baer modules). However, if $H \in \mathcal B$, then (\ref{dagger}) splits, so by a classic result of Bass (see \cite[28.2]{AF}), there exists $0 < n < \omega$ and an endomorphism $k$ of $X$ such that $g^n = kg^{n+1}$. Let $g^\prime : \mbox{Im }g^n \to \mbox{Im }g^{n+1}$ and $k^\prime : \mbox{Im }g^{n+1} \to 
\mbox{Im }g^n$ be the restrictions of $g$ and $k$, respectively. Then $g^\prime k^\prime = 1$, so $\mbox{Ker }k \cap \mbox{Im }g^{n+1} = 0$. Let $\bar{k} : \mbox{Coker }g^{n+1} \to 
\mbox{Coker }g^n$ be the map induced by $k$. Note that $\mbox{Coker }g^i \cong G^{(i)}$ for all $i < \omega$, because $\Ext 1RGG = 0$. Since $\mbox{End }G$ is a skew-field, we infer that $\mbox{Ker } \bar{k}$ is a direct summand in $\mbox{Coker }g^{n+1}$ and $\mbox{Ker }\bar{k} \cong G^{(j)}$ for some $j < \omega$ (cf.\ \cite[12.7]{AF}). By the above, $\mbox{Ker }\bar{k} = (\mbox{Ker }k \oplus \mbox{Im }g^{n+1})/\mbox{Im }g^{n+1} \cong \mbox{Ker }k$. Since $G \notin \mathcal B$, $G$ does not embed into the Baer module $X$, whence $\mbox{Ker }k = 0$. But then $\bar{k}$ yields an isomorphism of $G^{(n+1)}$ onto $G^{(n)}$, a contradiction.   \end{proof}

\section{Locally $\mathcal F$-free modules and approximations}

Finally, we are going to show that in a number of cases, the class of all locally $\mathcal F$-free modules does not provide for precovers. Suprisingly, the phenomenon spans all countable hereditary artin algebras of infinite representation type. Since each deconstructible class closed under extensions is precovering (cf.\ Lemma \ref{filt}), we obtain thus the non-deconstructibility of locally $\mathcal F$-free modules in that setting.

Again we will use tilting theory, the key property being the failure of the $\sum$-pure-split property of the tilting module, that is, the difference between the induced tilting cotorsion pair and its closure. 

\begin{theorem}\label{tiltprec} Let $R$ be a countable ring, $T$ a tilting module which is not $\sum$-pure-split, $(\mathcal A, \mathcal B)$ the tilting cotorsion pair corresponding to $T$, and $(\bar{\mathcal A}, \bar{\mathcal B})$ the closure of $(\mathcal A,\mathcal B)$. Let $\mathcal F$ denote the class of all countably presented modules from $\mathcal A$. 

Assume that $\mathcal L \subseteq \mathcal P_1$, $\mathcal L$ is closed under direct summands, and $N \in \mathcal L$ whenever $N$ is a module such that there exists $L \in \mathcal L$ with $N \subseteq L$ and $L/N \in \bar{\mathcal A}$. 

Then the class $\mathcal L$ is not precovering.     
\end{theorem}

\begin{proof} We have $\mathcal L \subseteq \varinjlim \mathcal F = \varinjlim \mathcal A  = \bar{\mathcal A}$ (see \cite[8.40 and 13.46]{GT}). As $R$ is countable, $\bar{\mathcal A} = \Filt (\mathcal C )$, where $\mathcal C$ is the class of all countably presented modules from $\bar{\mathcal A}$ (see \cite[6.17]{GT}). Since $T$ is not $\sum$-pure-split, $\mathcal A \subsetneq \bar{\mathcal A}$, so there exists a countably presented module 
$N \in \varinjlim _{\omega} \mathcal F \setminus \mathcal L$. Each such module $N$ is a Bass module for $\mathcal F$, so the Eklof lemma and Lemma \ref{griffith_counting} yield $\mathcal L ^\perp = \bar{\mathcal B}$.   
The rest of the proof is a generalization of the one for \cite[3.10]{SaT}.

Since $\bar{\mathcal B}$ is an enveloping class, there exists a short exact sequence $0 \to N \to \bar{B} \to \bar{A} \to 0$ where $\bar{B} \in \bar{\mathcal B}$ and $\bar{A} \in \bar{\mathcal A}$. Then also $\bar{B} \in \bar{\mathcal A}$, but $\bar{B} \notin \mathcal L$ (otherwise $N \in \mathcal L$ by our assumption on the class $\mathcal L$, since $\bar{B}/N \in \bar{A}$). 

We claim that there does not exist any $\mathcal {L}$--precover of the module $\bar{B}$. Assume $f: L \to \bar{B}$ is such a precover. Since $\bar{B} \in \bar{\mathcal A} = \varinjlim \mathcal L$, $f$ is surjective.

Let $\mathcal X = \{ M \in \rmod R \mid M + \Ker f = L \} \cap \bar{\mathcal B}$, and $\mathcal Y = \{ M \in \rmod R \mid  M + \Ker f = L \} \setminus \bar{\mathcal B}$.

Note that $X \cap \Ker f \notin \bar{\mathcal B}$ for each $X \in \mathcal X$, since otherwise $\bar{B} \in \bar{\mathcal A}$ would give $X = (X \cap \Ker f ) \oplus X^\prime$ for some $X^\prime \cong L/\Ker f \cong \bar{B}$, whence $L = \Ker f \oplus X^\prime$, and $\bar{B} \in \mathcal L$ (because $\mathcal L$ is closed under direct summands), a contradiction.

Since $\mathcal L ^\perp = \bar{\mathcal B}$, for each $X \in \mathcal X$, there exists $L_X \in \mathcal L$, such that $\Ext 1R{L_X}{X \cap \Ker f} \neq 0$. As $X \in \bar{\mathcal B}$ and $\bar{B} \cong X/(X \cap \Ker f)$, we can choose $f_X \in \Hom R{L_X}{\bar{B}}$ so that $f_X$ does not factorize through $f \restriction X$, and let $\tilde L = L \oplus \bigoplus _{X \in \mathcal X} L_X \in \mathcal L$.

Similarly, for each $Y \in \mathcal Y$ there exists a module $L_Y \in \mathcal L$ with $\Ext 1R{L_Y}Y \neq 0$. Let $L^\prime = \bigoplus _{Y \in \mathcal Y}L_Y \in \mathcal L$. 

Let $\varepsilon: \tilde L \hookrightarrow Z$ be a special $\{ L^\prime \}^\perp$--preenvelope of $\tilde L$. Since $\mbox{Coker}(\varepsilon)\in {}^\perp (\{ L^\prime \}^\perp) \subseteq \mathcal L$, also $Z \in \mathcal L$.

As $\bar{B} \in \bar{\mathcal B}$, we can factorize the epimorphism $(f \oplus \bigoplus _{X \in \mathcal X} f_X): \tilde L \to \bar{B}$ through $\varepsilon$, and obtain an epimorphism $h \in\Hom RZL$ such that $f \oplus \bigoplus _{X \in \mathcal X} f_X = h \varepsilon$. Since $f$ is a $\mathcal L$--precover of $\bar{B}$, there exists $g \in \Hom RZL$ such that $fg = h$. In particular, $I + \Ker f = L$, where $I = \mbox{Im} (g)$.

Since the projective dimension of $L^\prime \in \mathcal L$ is $\leq 1$, the class $\{ L^\prime \}^\perp$ is closed under homomorphic images, and hence contains $I$. By the definition of $L^\prime$ above, necessarily $L^\prime \in \mathcal X$. However, $f_I = h (\varepsilon\restriction L_I) = (f \restriction I) g (\varepsilon\restriction L_I)$, in contradiction with the choice of the homomorphism $f_I$ above. This proves our claim. 
\end{proof}

Again, we will have three corollaries: for flat Mittag-Leffler modules over non-perfect rings, for modules over Dedekind domains, and for locally Baer modules over hereditary artin algebras of infinite representation type. The first one has recently been proved in \cite{BS}, the other two are new:

\begin{corollary}\label{MLBS} 
Let $R$ be a countable non-right perfect ring. Then the class of all flat Mittag-Leffler modules is not precovering.  
\end{corollary}
\begin{proof} This follows from Theorem \ref{tiltprec} by taking $T = R$. Indeed, $T$ is not $\sum$-pure-split because $R$ is not right perfect, and  $\mathcal L$ is the class of all flat Mittag-Leffler modules which is closed under direct summands and consists of modules of projective dimension $\leq 1$ (in fact, since $R$ is countable, each flat module has projective dimension $\leq 1$). The condition of $N \in \mathcal L$ in the case when $N \subseteq L^\prime \in \mathcal L$ and $L^\prime/N \in \bar{\mathcal A}$ follows from the fact that the class of all flat Mittag-Leffler modules is closed under pure submodules.      
\end{proof}

\begin{corollary}\label{Dedekindnprec} Let $R$ be a countable Dedekind domain and $P$ a non-empty set of maximal ideals in $R$. Consider the tilting module $T_P = R_P  \bigoplus_{q \in \mspec R \setminus P} E(R/q)$, where $R_P = \bigcap_{p \in P} R_{(p)}$. Then $T_P$ is not $\sum$-pure-split. 

Let $(\mathcal A _P, \mathcal B _P)$ be the tilting cotorsion pair induced by $T_P$, $\mathcal F _P$ be the class of all countably presented modules from $\mathcal A _P$, and $\mathcal L _P$ the class of all locally $\mathcal F _P$-free modules. Then $\mathcal L _P$ is not precovering.  
\end{corollary}
\begin{proof} Let $Q$ be the quotient field of $R$. Then $R_P$ is a subring of $Q$ which is not perfect, hence there is a short exact sequence of the form $0 \to R_P^{(\omega)} \to  R_P^{(\omega)} \to N \to 0$ which is pure, but not split, so $T_P$ is not $\sum$-pure-split.

Since the class $\mathcal A _P$ is closed under submodules, so is $\mathcal L _P$, and Theorem \ref{tiltprec} applies. 
\end{proof}

\begin{corollary}\label{locbaernprec} 
Let $R$ be a countable finite dimensional hereditary algebra of infinite representation type. Then the class of all locally Baer modules is not precovering.  
\end{corollary}
\begin{proof} This follows from Theorem \ref{tiltprec} by taking $T = L$, the Lukas tilting module. Recall that $L$ is not $\sum$-pure-split, $\mathcal F = \mathcal C$ is the class of all countably presented Baer modules, that is, the countably $\mathcal R$-filtered modules (where $\mathcal R$ denotes the class of all finitely presented preprojective modules). Moreover, $\mathcal L$ is the class of all locally Baer modules. Since $\mathcal R$ is closed under submodules, so is $\mathcal L$, and Theorem \ref{tiltprec} applies.   
\end{proof}

\begin{remark}\label{different} Notice that in the tame hereditary case, the classes $\mathcal L$ constructed in Corollaries \ref{Lukas} and \ref{locbaernprec} are different: in the former one, the countably generated elements of $\mathcal L$ are countable direct sums of copies of $L$, while in the latter one, the countably generated modules in $\mathcal L$ are exactly the countably generated Baer modules. So the former class contains no non-zero finitely generated modules, while the finitely generated modules in the latter are exactly all the finitely generated preprojective modules. 
\end{remark}

\end{document}